\documentclass[11pt]{article}
\usepackage{graphicx}
\usepackage{amsmath}
\usepackage{amsthm}
\usepackage{amsfonts}
\usepackage{amssymb}
\usepackage{algorithm}
\usepackage{algpseudocode}
\usepackage{fullpage}

\newtheorem{theorem}{Theorem}

\newtheorem{lemma}[theorem]{Lemma}

\theoremstyle{definition}

\newtheorem{example}[theorem]{Example}

\def\beqs{\begin{eqnarray*}}
\DeclareMathOperator\lcs{lcs}
\DeclareMathOperator\simc{sc}

\title{Avoiding approximate repetitions with respect to the longest common subsequence distance}

\author{Serina Camungol and Narad Rampersad\\
Department of Mathematics and Statistics, University of Winnipeg\\
515 Portage Ave., Winnipeg, Manitoba, R3B 2E9 Canada\\
\texttt{\{serina.camungol, narad.rampersad\}@gmail.com}}

\begin{document}
\maketitle
\begin{abstract}
  Ochem, Rampersad, and Shallit gave various examples of infinite
  words avoiding what they called approximate repetitions.  An
  approximate repetition is a factor of the form $xx'$, where $x$ and
  $x'$ are close to being identical.  In their work, they measured the
  similarity of $x$ and $x'$ using either the Hamming distance or the
  edit distance.  In this paper, we show the existence of words
  avoiding approximate repetitions, where the measure of similarity
  between adjacent factors is based on the length of the longest
  common subsequence.  Our principal technique is the so-called
  ``entropy compression'' method, which has its origins in Moser and
  Tardos's algorithmic version of the Lov\'asz local lemma.

(2010 Mathematics Subject Classification: 68R15)
\end{abstract}

\section{Introduction}
A now classical result of Thue \cite{Thu06} showed the existence of an
infinite word over a $3$-letter alphabet avoiding \emph{squares}; that
is, factors of the form $xx$.  Ochem, Rampersad, and Shallit
\cite{ORS08} generalized the work of Thue by constructing infinite
words over a finite alphabet that avoid factors of the form $xx'$,
where $x$ and $x'$ are close to being identical.  In most of their
work, the closeness of $x$ and $x'$ was measured using the Hamming
distance; they also have some results where the edit distance was used
instead.  Here, we measure the closeness of two words based on the
length of their longest common subsequence.

The most common metrics used to measure the distance between strings
are the edit distance, the Hamming distance, and the longest common
subsequence metric.  The edit distance is the most general: it is
defined as the smallest number of single-letter insertions, deletions,
and substitutions needed to transform one string into the other.  The
other two distances can be viewed as restricted versions of the edit
distance:  the Hamming distance (between strings of the same length)
is the edit distance where only the substitution operation is permitted;
the longest common subsequence metric allows only insertions and
deletions.

The study of the longest common subsequence of two (or several)
sequences has a lengthy history (which, at least initially, was
motivated by the biological problem of comparing long protein or
genomic sequences).  For example, in 1975 Chv\'atal and Sankoff \cite{CS75}
explored the following question:  Given two random sequences of length
$n$ over a $k$-letter alphabet, what is the expected length of their
longest common subsequence?  Questions concerning longest common
subsequences in words continue to be studied to this day (see the
recent preprint \cite{BZ13}, for example).

Ochem, Rampersad, and Shallit \cite{ORS08} previously studied the
avoidability of approximate squares with respect to Hamming distance
and edit distance.  Using the longest common subsequence metric has
not yet been done, so it is the aim of this paper to consider the
avoidability of approximate squares with respect to this measure of
distance.

Our main result is non-constructive---indeed it seems to be quite
difficult to find explicit constructions for words avoiding the kinds
of repetitions we consider here---and is based on the the so-called
``entropy compression'' method, which originates from Moser and
Tardos's algorithmic version of the Lov\'asz local lemma \cite{MT10}.
This method has recently been applied very successfully in
combinatorics on words, for instance by \cite{GKM11} and \cite{GKW11}.
Ochem and Pinlou \cite{OP13} also recently resolved a longstanding conjecture of
Cassaigne using this method (this was also accomplished independently
by Blanchet-Sadri and Woodhouse \cite{BSW13} using a different method).

\section{Measuring similarity}

The definitions given in this section are essentially those of Ochem
et al., except that they are based on the longest common
subsequence distance rather than the Hamming distance.

For words $x, x'$, let $\lcs(x,x')$ denote the length of a longest
common subsequence of $x$ and $x'$.
For example, $\lcs({\tt 0120}, {\tt 1220}) = 3$.
Given two words $x, x'$ of the same length,
we define their {\it similarity} $s(x,x')$ by
$$ s(x,x') := \frac{\lcs(x,x')}{|x|} .$$
For example, $s({\tt 20120121}, {\tt 02102012}) = 3/4$.

The {\it similarity coefficient} $\simc(z)$ of a finite word $z$ is defined to be
$$\simc(z) := \max \{s(x,x') : xx' {\rm\ a\ subword\ of\ } z \text{ and } |x|=|x'|\}.$$  If $\simc(z) = \alpha$, we
say that $z$ is {\it $\alpha$-similar}.  If ${\bf z}$ is an infinite word, then its similarity coefficient is defined by
$$\simc({\bf z}) := \sup \{s(x,x') : xx' {\rm\ a\ subword\ of\ } {\bf z} \text{
  and } |x|=|x'|\}.$$ Again, if $\simc({\bf z}) = \alpha$ then we say that
${\bf z}$ is $\alpha$-similar.

\section{Infinite words with low similarity}
\label{sec:algorithm}

Our main result is the following:

\begin{theorem}\label{main_theorem}
Let $0<\alpha<1$ and let $k > 16^{1/{\alpha}}$ be an integer.  Then there exists an infinite word ${\bf z}$ over
an alphabet of size $k$ such that $\simc({\bf z}) \leq \alpha$.
\end{theorem}

To prove this, we follow the method of Grytczuk, Kozik, and Witkowski \cite{GKW11}.
We begin by defining a randomized algorithm which attempts to
construct a word of length $n$ with similarity coefficient at most
$\alpha$ by a sort of backtracking procedure.

\begin{algorithm}
\caption{Choose a sequence with similarity coefficient at
most $\alpha$}
Input : $n$, $k$, $\alpha$
\begin{algorithmic}[1]
\State $S = \emptyset, i = 1$
\While{$i \leq n$} 
\State randomly choose $y \in \{ 1, \ldots, k \}$ and append $y$ to $S$
\State let $s_{i}$ denote the $i^{th}$ element of $S$
\If{$\simc(s_{1}s_{2} \cdots s_{i}) \leq \alpha$}  
set $i$ to $i+1$
\Else \, $s_{1}s_{2} \cdots s_{i}$ is $\beta-$similar, $\beta>\alpha$, and contains a subword $xx^{\prime}$ such that $|x|=|x^{\prime}|=\ell$, $\ell \leq\frac{i}{2}$ and $s(x, x^{\prime})=\beta$,
say $x=s_{t+1}s_{t+2} \cdots s_{t+\ell}$ and $x^{\prime}=s_{t+\ell+1}s_{t+\ell+2} \cdots s_{t+2\ell}$, where $t+2\ell = i$.
\For {$t+\ell+1 \leq j \leq t+2\ell$}
\State delete $s_{j}$
\EndFor
\State set $i =t+\ell+1$
\EndIf
\EndWhile
\end{algorithmic}
\end{algorithm}

The algorithm generates consecutive terms of a sequence $S$ by choosing symbols at random (uniformly and independently).  Every time a $\beta-$similar subword $xx^{\prime}$ is created, the algorithm erases $x^{\prime}$, to ensure that the $\beta-$similar subword is deleted. 

It is easy to see that the algorithm terminates after a word of length
$n$ with similarity coefficient at most $\alpha$ has been produced.
The general idea is to prove the algorithm cannot continue forever
with all possible evaluations of the random inputs.

Fix a real number $\alpha$.  We will show that for every positive
integer $n$ there exists a word of length $n$ with similarity
coefficient at most $\alpha$.  The existence of an infinite word with
the same property then follows by a standard compactness argument.

Let $n$ be a positive integer, and suppose for the sake of
contradiction that the algorithm fails to produce a sequence of length
$n$; this means the algorithm continues forever.  We are going to count
the possible executions of the algorithm in two ways:

Suppose the algorithm runs for $M$ steps.  By ``step'' we mean appending a letter to the sequence $S$ (which only happens in
line 3).  Let $r_{1}, r_{2}, \ldots, r_{M}$ be the sequence of values
chosen randomly and independently in the first $M$ steps of the
algorithm.  Each $r_{j}$, $1 \leq j \leq M,$ can take $k$ different
values, thus there are $k^{M}$ such sequences.

The second way of counting involves analysing the behaviour of the algorithm.
The following are four elements, defined for every fixed evaluation of the first $M$ random choices of the algorithm.

\begin{itemize}
  \item A route $R$ in the upper right quadrant of the Cartesian plane, going from coordinate $(0,0)$ to coordinate $(2M, 0)$, with possible moves $(1,1)$ and $(1,-1)$, which never goes below the axis $y=0$
  \item A sequence $X$ whose elements correspond to the peaks on the route $R$, where peak is defined as a move $(1,1)$ followed immediately by a move $(1,-1)$
  \item A sequence $Y$ where elements of $Y$ correspond to elements of $X$
  \item A sequence $S$ produced after $M$ steps of the algorithm
\end{itemize}

We call the quartet $\{R, X, Y, S\}$ a log and we encode consecutive steps of the Algorithm in the following way:

Each time the algorithm appends a letter to the sequence $S$, we append a move $(1,1)$ to the route $R$ and everytime an $s_{i}$ is nullified we append $(1,-1)$. Every down step $(1, -1)$ corresponds to an up-step $(1,1)$ so we never reach below the $y$-axis. At the end of computations we add to the route $R$ one down-step for each element of $S$ which was not deleted at any point in the algorithm, bringing us to the point $(2M, 0)$. If a $\beta-$similar word is created, say $xx^{\prime}$, we append a similar version of $x^{\prime}$ to $X$, but replace the elements of the longest common subsequence of $x$ and $x^{\prime}$ with the symbol $*$.  At the end of computations we append to $X$ enough $*$'s so that $|X| = M$ .   We construct $Y$ similarly, but using $x$ instead of $x^{\prime}$ and placing $0$'s in positions that are not part of the longest common subsequence of $x$ and $x^{\prime}$.  Lastly, $S$ is the sequence produced by the Algorithm after making $M$ random selections from $\{ 1, \ldots, k\}$.

\begin{example}
For example, let us choose $\alpha = \frac{37}{50}$.  Then $\lceil 16^{\frac{50}{37}} \rceil = 43$ and we have alphabet $\{1, \ldots, 43\}$ and log $\{ R = \emptyset, X=\emptyset, Y=\emptyset, S=\emptyset \}$.  Suppose we create the word ${\tt 12023431354}$ after 11 steps of the algorithm.  Each of our steps avoids creating a $\beta-$similar word, so at each step we append $(1,1)$ to $R$ and the randomly selected letter to $S$.  Thus we have:
$$\{ R = (1,1)^{11}, X = \emptyset, Y= \emptyset,  S = {\tt 12023431354}\}.$$
Suppose in the 12th step of the algorithm we append `3' to $S$, then our log becomes:
$$\{ R = (1,1)^{12}, X = \emptyset, Y= \emptyset,  S = {\tt 120234313543}\}.$$
Observe that the factor $xx^{\prime} = {\tt 34313543}$ is $\frac{3}{4}-$similar, where $x={\tt 3431}$, ${x^{\prime}={\tt 3543}}$ and the longest common subsequence of $x$ and $x^{\prime}$ is {\tt 343}.  As $\frac{3}{4} > \frac{37}{50}$, we replace the longest common subsequence elements of $x$ and $x^{\prime}$ with $*$'s and we append $*5**$ to $X$ and $***0$ to $Y$.
We then delete $x^{\prime}$ and append to $R$ a $(1,-1)$ for each deleted element.  This results in the following log:
$$\{ R = (1,1)^{12}(1,-1)^{4}, X = *5**, Y= ***0,  S = {\tt 12023431}\}.$$
This is where we conclude our example.
\end{example}

\begin{lemma}\label{algorithm_check}
Every log corresponds to a unique sequence $r_{1}, r_{2}, ..., r_{M}$ of the first $M$ values chosen randomly and independently in some execution of the Algorithm.
\end{lemma}

\begin{proof}
Before we decode $r_{1}, r_{2}, ..., r_{M}$, we do some preparatory analysis.  We construct a sequence $D=\{d_{1}, d_{2}, \ldots, d_{p}\}$, 
corresponding to the lengths of consecutive down-steps, $(1, -1)$, of $R$.  We then delete the last $M - (d_{1} + d_{2}+ \cdots +d_{p})$ $*$'s from $X$.  Doing so results in a sequence  $X=\{x_{1}x_{2} \cdots x_{N}\}$ where each element of $X$ corresponds to the down-steps of $R$, it then follows that $d_{1} + d_{2}+ \cdots +d_{p} = |X| = N$.  We split up $X$ to form a new set, $X^{\prime}$, where the length of each partitioned block corresponds to an element of $D$, so that $$X^{\prime}=\{x_{1}x_{2} \cdots x_{d_{1}}, \; x_{d_{1}+1}x_{d_{1}+2} \cdots x_{d_{1}+d_{2}}, \; \ldots, x_{N-d_{p}+1}x_{N-d_{p}+2} \cdots x_{N}\}.$$  Note that $|X^{\prime}|=|D|$, so that every element in $X^{\prime}$ couples with an element in $D$.  We do the same process for the sequence $Y$, obtaining a new sequence $$Y^{\prime}=\{y_{1}y_{2} \cdots y_{d_{1}}, \; y_{d_{1}+1}y_{d_{1}+2} \cdots y_{d_{1}+d_{2}}, \; \ldots, y_{N-d_{p}+1}y_{N-d_{p}+2} \cdots y_{N}\}.$$

Next we use information from route $R$ to determine which $s_{i}$, $1 \leq i \leq n$, were not nullified at each step of the Algorithm and to find the coordinates of the blocks which were nullified at step $(8)$ of the Algorithm.  Notice that appending some letter from $\{ 1, \ldots, k \}$ to $S$ corresponds to some up-step $(1,1)$ on the route $R$, while deleting an $s_{i}$ corresponds to some down-step $(1,-1)$ on the route $R$.  We analyse the route $R$, starting from the point $(0,0)$ to the point $(2M, 0)$.  Assume the first peak occurs between the $j^{th}$ and $(j+1)^{th}$ step.  As this is the first time we erase elements $s_{i}$ and we know that $s_{1}, \ldots, s_{j}$ are the only non-deleted elements at this point.  From the number of down-steps on $R$ we deduce the length of the nullified similar block, say there are $d_{1}$ down-steps, and remember that for this peak we deleted $s_{j-d_{1}+1}, s_{j-d_{1}+2}, \ldots, s_{j}$.  Now again each up-step on $R$ denotes appending some value of $\{ 1, \ldots, k \}$ to $S$.  
Continuing on in this manner, we are able to determine exactly which position was set last as we reach the next peak.  From this information it is easy to determine which positions were nullified as a result of erasing the repetition.  We repeat these operations until we reach the end of the route $R$.


After these preparatory measures we are ready to decode $r_{1}, r_{2}, \ldots, r_{M}$.  We consider the sequence $R$ in reverse order, from the point $(2M,0)$ to the point $(0,0)$, modifying the sequences $X^{\prime}$ and $Y^{\prime}$ from the preparatory step and the final sequence $S$. We use information encoded in $S$, $X^{\prime}$ and $Y^{\prime}$ as well as knowledge from the preparatory step.  

First we consider the up-steps $(1,1)$ and note than each up-step corresponds to some $r_{i}$.  In the preparatory analysis we determined the indices of elements $r_{i}$ in $S$ so, each time there is an up-step on $R$, we assign to $r_{j}$ a value from appropriate $s_{i}$ (where $i$ was determined in the preparatory step), and delete $s_{i}$.  

Now we consider the down-steps of $R$.  At the beginning of $R$ there is some number of down-steps corresponding to the last non-deleted elements of $S$ (the elements added at the end of computations), we skip these elements and move on.  The first block of down-steps that follow an up-step has length $d_{p}$ and corresponds to the last element of $X^{\prime}$, say $X^{\prime}_{N}$ as well as the last element of $Y^{\prime}$, say $Y^{\prime}_{N}$.  We compare the sequence $s_{i-d_{p}}, s_{i-d_{p}+1}, \ldots, s_{i-1}$, to $Y^{\prime}_{N}$, and the sequence $s_{i}, s_{i+1}, \ldots, s_{i+d_{p}-1}$ to $X^{\prime}_{N}$, where $s_{i}$ is the first element of the erased similar block determined in the preparatory step.

Together, the indexed elements of $s_{i-d_{p}}, s_{i-d_{p}+1}, \ldots, s_{i-1}$ that correspond to the $*$ elements of $Y^{\prime}_{N}$ form the nullified longest common subsquence, call this sequence $LCS$.  $LCS$ also corresponds to the $*$ elements of $X^{\prime}_{N}$, so we can replace the $*$ elements of $X^{\prime}_{N}$.  Now $X^{\prime}_{N}$ is the last deleted block of the Algorithm.  We now replace $s_{i}, s_{i+1}, \ldots, s_{i+d_{p}-1}$ with the newly changed $X^{\prime}_{N}$, altering the sequence $S$.  Continuing in this manner, we are able to retrieve all deleted blocks of the Algorithm: they are the elements of the sequence $X^{\prime}$.
\end{proof}

We have just shown that there is an injective mapping between the set of all sequences of randomly chosen values during the execution of the algorithm and the set of all logs.  Consequently, the number of different logs is always greater or equal to the number of possible sequences $r_{1}, r_{2}, r_{3}, \ldots, r_{M}$.  We now derive an upper bound for the number of possible logs.

The number of possible routes $R$, of length $2M$ and possible moves $(1,1)$ and $(1,-1)$, in the upper right quadrant of the Cartesian plane is the $M^{th}$ Catalan number $C_{M}$.

To count $X$ we first note that $|X|=M$ and that each deleted factor $x^{\prime}$ has (strictly) more than $\alpha|x^{\prime}|$ positions $*$, so it follows that $X$ has more than $\alpha M$ positions $*$.  Let $j$ be the number of $*$'s in $X$.  There are $k$ choices for the $M-j$ non-$*$ positions in $X$, so there are $\binom{M}{j}k^{M-j}$ possibilities for $X$.  Now if $X$ has $j$ $*$'s, then so does $Y$, and the remaining positions in $Y$ are $0$'s.  Thus, there are $\binom{M}{j}$ possibilities for $Y$, and hence $\binom{M}{j}^2k^{M-j}$ possibilities for the pair $(X,Y)$.  Summing over all $j$, we conclude that there are
$${\displaystyle\sum\limits_{j=\lceil \alpha M \rceil}^{M} \binom{M}{j}^{2} k^{M-j}}$$
possibilities for the pair $(X,Y)$.

The sequence $S$ consists of $\leq n$ elements of value between $1$ and $k$, so there are $\frac{k^{n+1}-1}{k-1}$ possible sequences $S$.

Multiplying these individual bounds together brings us to the conclusion that the number of possible logs is at most
$$\frac{k^{n+1}-1}{k-1}C_{M}{\displaystyle\sum\limits_{j=\lceil \alpha M \rceil}^{M} \binom{M}{j}^{2} k^{M-j}}.$$

Comparing with the number $k^M$ of possible choices for the sequence $r_1,r_2,\ldots,r_M$ we get the inequality
$$k^M \leq \frac{k^{n+1}-1}{k-1}C_{M}{\displaystyle\sum\limits_{j=\lceil \alpha M \rceil}^{M} \binom{M}{j}^{2} k^{M-j}}.$$

Asymptotically, the Catalan numbers $C_{M}$ satisfy
$$C_{M} \sim \frac{4^{M}}{M\sqrt{\pi M}},$$
and $\binom{M}{j} < 2^{M}$, which implies that
$$k^M \ll \frac{k^{n+1}-1}{k-1}\frac{4^{M}}{M\sqrt{\pi
    M}}{\displaystyle\sum\limits_{j=\lceil \alpha M \rceil}^{M}
  (2^{M})^{2} k^{M-j}}.$$

Simplifying we get that
\begin{align}
k^M & \ll \frac{k^{n+1}-1}{k-1}\frac{4^{M}}{M\sqrt{\pi M}} 4^{M}
{\displaystyle\sum\limits_{j=\lceil \alpha M \rceil}^{M} k^{M-j}} \nonumber\\
& = \frac{k^{n+1}-1}{k-1}\frac{16^{M}}{M\sqrt{\pi M}} {\displaystyle\sum\limits_{j=0}^{M - \lceil \alpha M \rceil} k^{j}}\nonumber\\
& = \frac{16^{M}}{M\sqrt{\pi M}}\frac{(k^{n+1}-1)(k^{M - \lceil \alpha M \rceil +1}-1)}{(k-1)^2} \nonumber\\
& \leq k^{n+2}\frac{16^{M}}{M\sqrt{\pi M}} \frac{k^{M(1-\alpha)}}{(k-1)^{2}} \label{logs_ineq}
\end{align}

We claim that when $k > 16^{1/\alpha}$, the right hand side of
  \eqref{logs_ineq} is $o(k^{M})$ and therefore for large enough $M$
  the inequality \eqref{logs_ineq} cannot hold.
This contradiction implies that for some specific choices of $r_{1},
r_{2}, \ldots$ the Algorithm stops (i.e., produces a word of length
$n$ with similarity coefficient at most $\alpha$).

In order to verify the choice of $k$ needed to obtain the
contradiction described above; we need 
$k^{n+2}\frac{16^{M}}{M\sqrt{\pi M}} \frac{k^{M(1-\alpha)}}{(k-1)^{2}}$
to be $o(k^{M})$.
We can disregard the term $\frac{k^{n+2}}{M\sqrt{\pi M}(k-1)^2}$, as it
approaches 0 as M approaches infinity, and concentrate on the term
$16^{M}k^{M(1-\alpha)}$.  We wish to determine the values of $k$ that
satisfy
$$\lim\limits_{M \to \infty} \frac{16^{M}k^{M(1-\alpha)}}{k^M} = 0,$$
or more simply
$$\lim\limits_{M \to \infty} 16^{M}k^{-M\alpha} = 0.$$
This will hold if 
\[
M\log16 - \alpha M\log k < 0,
\]
which holds whenever
\[
\log k > \frac{\log16}{\alpha},
\]
or, in other words, whenever
\[
k > 16^{1/\alpha}.
\]
This completes the proof of Theorem~\ref{main_theorem}.

\section{Similarity coefficients for small alphabets}

Almost certainly, the bound of $16^{1/\alpha}$ for the size of the
alphabet needed to obtain an infinite word with similarity coefficient
at most $\alpha$ is far larger than the true optimal alphabet size.
For example, for $\alpha = 0.9$ we get an alphabet size of $22$, which
is surely much larger than necessary.  In this section we investigate
the following question: Given an alphabet $\Sigma$ of size $k$, what
is the smallest similarity coefficient possible over all infinite
words over $\Sigma$?  Implementing an algorithm similar to that of
Section~3 allows us to get an idea of which values of $\alpha$, $0 <
\alpha < 1,$ are avoidable and unavoidable.  Given a similarity
coefficient $\alpha$ to avoid, a length $n$, and an alphabet size $k$,
the algorithm starts at 0 and appends letters until a word of length
$n$ with similarity coefficient $< \alpha$ is obtained.  If a factor
with similarity coefficient $\geq \alpha$ is created, the last
appended letter is deleted.  If appending no other letter avoids
$\alpha$, the algorithm deletes yet another letter, and so on and so
forth.  The algorithm continues until a word of length $n$ is
produced.  If no word of length $n$ avoids $\alpha$, the algorithm
returns the longest word avoiding $\alpha$.  If, on the other hand,
the algorithm produces words with similarity coefficient $< \alpha$
for longer and longer values of $n$, then we take this as evidence
that there exists an infinite word over a $k$-letter alphabet with
similarity coefficient $< \alpha$.  We performed this computation for
various alphabet sizes, and the results can be found in
Table~\ref{backtrack_table}.

\begin{table}
\begin{center}
\begin{tabular}{| c | c |}
\hline
Alphabet size & Similarity coefficient\\ \hline
3 & $0.888 < \alpha < 0.901$ (?) \\ \hline
4 & $0.690 < \alpha < 0.760$ (?) \\ \hline
5 & $0.590 < \alpha < 0.700$ (?) \\ \hline
6 & $0.500 < \alpha < 0.650$ (?) \\ \hline
7 & $0.450 < \alpha < 0.650$ (?) \\ \hline
8 & $0.400 < \alpha < 0.570$ (?) \\ \hline
\end{tabular}
\end{center}
\caption{Results of the backtracking algorithm \label{backtrack_table}}
\end{table}

For each lower bound reported in the table, we are certain that there
does not exist an infinite word with this similarity coefficient.
However, the upper bounds are only conjectural: the backtracking
algorithm described above produces long words with similarity
coefficient less than the stated bound, but we have no conclusive
proof that an infinite word exists.

In fact, we cannot produce a single explicit construction (with proof)
of an infinite word with similarity coefficient less than 1.  However,
computer calculations suggest that the so-called \emph{Dejean words}
seem to have fairly low similarity (though not nearly as low as the
values given in Table~\ref{backtrack_table}).  We now report the
results of our computer calculations on the words constructed by
Moulin-Ollagnier \cite{MO92} in order to verify Dejean's Conjecture
for small alphabet sizes.  For each alphabet size $k=3,\ldots,11$,
Moulin-Ollagnier constructed an infinite word over a $k$-letter
alphabet.  Each such word verified a conjecture of Dejean \cite{Dej72}
concerning the repetitions avoidable on a $k$-letter alphabet.  The
details of Dejean's Conjecture, and the precise nature of
Moulin-Ollagnier's construction can be found in his paper.  In
Table~\ref{ollagnier_table}, we report the largest similarity
coefficient found among all factors of Moulin-Ollagnier's words, up to
a certain length.  In the table, ``Prefix length'' is the length of
the prefix of the infinite word that we examined.  ``Factor length''
is the maximum length of the factors of this prefix that we examined.
A `-' signifies a continuous increase in similarity coefficient as the
lengths of the factors increase.

\begin{table}
\begin{center}
\begin{tabular}{| c | c | c | c |}
\hline
Alphabet size & Similarity coefficient & Prefix length & Factor length \\ \hline
3 & - & 2401 & 500\\ \hline
4 & 11/12 & 912 & 500\\ \hline
5 & 16/19 & 9261 & 399\\ \hline
6 & 10/13 & 9261 & 312\\ \hline
7 & - & 5000 & 218\\ \hline
8 & 12/15 & 5000 & 445\\ \hline
\end{tabular}
\end{center}
\caption{Results of computer calculations on
  Moulin-Ollagnier's words \label{ollagnier_table}}
\end{table}

Two natural problems suggest themselves:
\begin{enumerate}
\item Determine the similarity coefficients of Moulin-Ollagnier's
  words.
\item For each alphabet size $k$, determine the least similarity
  coefficient among all infinite words over a $k$-letter alphabet.
\end{enumerate}

The second question is likely quite difficult.  Even an answer just for the
$3$-letter alphabet would be nice to have.

\end{document}